\documentclass[12pt,reqno]{amsart}
\usepackage[all,2cell]{xy}

\usepackage{amssymb}

\newtheorem{theorem}{Theorem}[section]

\newtheorem{lemma}[theorem]{Lemma}
\newtheorem{corollary}[theorem]{Corollary}

\theoremstyle{definition}
\newtheorem{definition}[theorem]{Definition}
\newtheorem{remark}[theorem]{Remark}

\numberwithin{equation}{section}

\DeclareMathOperator*{\Gal}{Gal}
\DeclareMathOperator*{\degree}{degree}

\newcommand{\CC}{\mathbb{C}}
 
\newcommand{\PP}{\mathbb{P}}

\newcommand{\GG}{\mathbb{G}}

\begin{document}

\title{Genuinely ramified maps and monodromy}

\author[I. Biswas]{Indranil Biswas}

\address{Department of Mathematics, Shiv Nadar University, NH91, Tehsil
Dadri, Greater Noida, Uttar Pradesh 201314, India}

\email{indranil.biswas@snu.edu.in, indranil29@gmail.com}

\author[M. Kumar]{Manish Kumar}

\address{Statistics and Mathematics Unit, Indian Statistical Institute,
Bangalore 560059, India}

\email{manish@isibang.ac.in}

\author[A.J. Parameswaran]{A. J. Parameswaran}

\address{School of Mathematics, Tata Institute of Fundamental
Research, Homi Bhabha Road, Bombay 400005, India}

\email{param@math.tifr.res.in}



\begin{abstract}
For any genuinely ramified morphism $f\, :\, Y\, \longrightarrow\, X$ between
irreducible smooth projective curves we prove that $\overline{(Y\times_X Y) \setminus \Delta}$
is connected, where $\Delta\, \subset\, Y\times_X Y$ is the diagonal. Using this result the following are proved:
\begin{enumerate}
\item If $f$ is further Morse then the Galois closure is the symmetric
group $S_d$, where $d\,=\, \text{degree}(f)$.

\item The Galois group of the general projection, to a line, of any smooth curve $X\,\subset\, \PP^n$ of degree $d$,
which is not contained in a hyperplane and contains a non-flex point, is $S_d$.
\end{enumerate}
\end{abstract}

\maketitle

\section{Introduction}

Let $f\,:\,Y\,\longrightarrow\, X$ be a generically smooth finite morphism of smooth projective 
curves over an algebraically closed field $k$. It is called genuinely ramified if the 
corresponding homomorphism of \'etale fundamental groups in surjective.
A map $Y\,\longrightarrow\, X$ is genuinely ramified if and only if $Y\times_X Y$ is connected
\cite{BP}. We prove the following improved version of it (see Theorem \ref{main}):

\textit{For any genuinely ramified morphism $Y\,\longrightarrow\, X$,
the closure of $$(Y\times_X Y) \setminus \Delta\, \subset\, Y\times_X Y$$ is connected,
where $\Delta$ is the diagonal in $Y\times_X Y$.}

Two applications of it are given.

\textit{If a genuinely ramified morphism $f\, :\, Y\, \longrightarrow\, X$ of degree $d$ is also Morse 
(see Section \ref{sec2} for Morse map), then the Galois group of the Galois closure of $f$ is 
$S_d$, where $d\,=\, \text{degree}(f)$.} (See Corollary \ref{S_d-cover}.)

\textit{Let $X\,\subset\, \PP^n$ be a smooth curve of degree $d$ and not contained in any hyperplane. Assume that
$X$ contains a non-flex point (see Definition \ref{non-flex}). Take a line $L$ 
in the dual projective space $\check{\PP}^n$ such that it intersects the dual variety $\check{X}$ transversely, every $H\,\in\, L\cap \check{X}$ is an ordinary tangent to $X$, and $X\cap H_1\cap H_2$ is empty for $H_1\,\ne\, H_2\,\in\, L$. Let
$f\,:\,X\,\longrightarrow\, L$ be the projection of $X$ from the codimension two linear subspace in
$\PP^n$ corresponding to $L$. Then the map $f$ is Morse and its Galois
closure has Galois group $S_d$.} (See Corollary \ref{cor-l}.)

Note that the later result (Corollary \ref{cor-l}) over the field of complex numbers is known and follows from the uniform position 
principle (see \cite{PS}). Again over $\CC$, the monodromy group of a degree $d$ cover $f\, :\, Y\, \longrightarrow\, \PP^1$ has been 
studied by many people (see \cite[Theorem 2]{GM}, \cite[Theorem 1 and 1']{Yo}, etc). Apart from extending some of these results to 
positive characteristic, our method allows computation of monodromy groups even when the base curve is not $\PP^1$.

\section{Morse map}\label{sec2}

Let $k$ be an algebraically closed field. Take connected smooth projective curves $X$ and $Y$ defined over $k$. A generically smooth
finite morphism $f\,:\,Y\,\longrightarrow\, X$ is called a finite covering (or a cover). Since here $X$ and $Y$
are integral schemes, the condition that the morphism $f$ is generically smooth is equivalent to $f$ being separable.
The cover $f$ is called Galois if $k(Y)/k(X)$ is a Galois extension. The cover $f$ is called \textit{Morse} if over every branch point
$x\, \in\, X$ the following two conditions hold: there is exactly one ramification point $y$ over $x$ and this $y$ is a double point.

Take a generically smooth finite morphism $f\,:\,Y\,\longrightarrow\, X$. Consider the fiber product $Y\times_X Y$. For $i\,=\,1,\, 2$,
let
\begin{equation}\label{g1}
q_i\, :\, Y\times_X Y\, \longrightarrow\, Y
\end{equation}
be the natural projection to the $i$-th factor. Consider the diagonal
\begin{equation}\label{g2}
\Delta \,\subset \, Y\times_X Y;
\end{equation}
so $\Delta$ is the image of the map $Y\, \longrightarrow\, Y\times_X Y$ given by the two copies of the identity map of
$Y$. Let
\begin{equation}\label{g0}
Y'\,\, :=\,\, \overline{(Y\times_X Y) \setminus \Delta}\,\, \subset\,\, Y\times_X Y
\end{equation}
be the closure of the complement $(Y\times_X Y) \setminus \Delta$ in $Y\times_X Y$.

\begin{lemma}\label{morse}
Let $f\,:\,Y\,\longrightarrow\, X$ be a Morse map. Then $Y'$ in \eqref{g0} is smooth. If $Y'$ is also connected, then the
restriction $q'_1\, :=\, q_1\big\vert_{Y'} \, :\, Y'\,\longrightarrow\, Y$ of $q_1$ (see \eqref{g1}) is again a Morse map.
\end{lemma}

\begin{proof}
Take any point 
$(y_1,\,y_2)\,\in\, Y'$, and denote the point $f(y_1)\,=\,f(y_2)\, \in\, X$ by $x$. If $y_1\,\ne\, y_2$, then $f$ is
\'etale at either $y_1$ or $y_2$. In that case, $(y_1,\, y_2)$ is a smooth point of $Y\times_X Y$, because
the map $q_i$ in \eqref{g1} is smooth at $(y_1,\, y_2)$ if $f$ is \'etale at $y_i$. Hence 
$(y_1,\,y_2)$ is a smooth point of $Y'$. If $y_1\,=\,y_2$, and $f$ is unramified at $y_1$, then again $(y_1,\, y_2)$ is a smooth
point of $Y\times_X Y$ for the same reason, hence it cannot be a point of $Y'$.

Assume that $y_1\,=\,y_2$, and that $f$ is ramified at $y_1$. Then by hypothesis $y_1$ is a double point of $f$. Note that $Y\times_X Y$ 
has two smooth local branches at $(y_1,\,y_1)$ with $\Delta$ and $Y'$ being the irreducible components through this point. Therefore, we 
conclude that $Y'$ is smooth at $(y_1,\, y_1)$. So $Y'$ is smooth.

For $y\, \in\, Y$, if $f(y)$ is not a branch point of $f$, then $(y,\, z)$ is not a branch point of $q'_1$ for every
$z\, \in\, f^{-1}(f(y))$.
Assume that $f(y)$ is a branch point of $f$. Let $z\,\in\, f^{-1}(f(y))$ be the unique double point of $f$ in $f^{-1}(f(y))$.
Then $(y,\, z)$ is the unique double point of $q'_1$ in $(q'_1)^{-1}(y)$ if $y\, \neq\, z$ proving that the map $q'_1$ is Morse. 
Note that if $y\,=\, z$, then again $y$ is not a branch point of $q'_1$.
\end{proof}

\begin{lemma}\label{2-transitive}
Let $f\,:\,Y\,\longrightarrow\, X$ be a cover of degree $d$ between irreducible smooth projective curves. If the complement
$(Y\times_X Y) \setminus \Delta$, where $\Delta\, \subset\, Y\times_X Y$ is the diagonal, is irreducible, then the Galois group of the
Galois closure of $f$ is a two-transitive subgroup of $S_d$. Conversely, if the Galois group of the Galois closure of $f$ is a
two-transitive subgroup of $S_d$, then $(Y\times_X Y) \setminus \Delta$ is irreducible.
\end{lemma}

\begin{proof}
Since $f$ is a covering of degree $d$, the function field $k(Y)$ of $Y$ satisfies
the condition
$$
k(Y)\,\,=\,\,k(X)[z]/F(z)
$$
for some irreducible separable polynomial $F$ 
of degree $d$. As the map $f$ is finite, every irreducible component of $Y\times_X Y$ dominates $X$. Hence the irreducible components of 
$Y\times_X Y$ are in bijection with the generic fiber of $Y\times_X Y\,\longrightarrow\, X$. The generic fiber corresponds to the ring extension
$k(X)\,\subset \, k(Y)\otimes_{k(X)}k(Y)$. But $$k(Y)\otimes_{k(X)}k(Y)\,\cong\, k(Y)[z]/(F(z)).$$
Let $\alpha$ be a root of $F(z)$ in $k(Y)$. Then $F(z)\,=\,(z-\alpha)F_1(z)$ in $k(Y)[z]$.

Note that $F_1(z)$ is irreducible if and only if $k(Y)[z]/(F(z))$ has only two closed points. The condition that $k(Y)[z]/(F(z))$
has only two closed points is evidently equivalent to the condition that $\Delta\,\subset\, Y\times_X Y$ as well as the closure of
$(Y\times_X Y) \setminus \Delta$ in $Y\times_X Y$ are both irreducible.
 
Let $\widetilde Y$ be the Galois closure of $f$. Note that the Galois group $G\,=\, \text{Gal}(\widetilde{Y}/X)$
acts on the roots of $F(z)$ in 
$k(\widetilde{Y})$. Since $F(z)$ is an irreducible separable polynomial of degree $d$, the above Galois
group $G$ is a transitive subgroup of $S_d$. Let $\beta$ and 
$\gamma$ be roots of $F_1(z)$. If $F_1(z)$ is irreducible, then there is an automorphism $\sigma$ of $k(\widetilde{Y})$ fixing $k(X)(\alpha)$ such 
that $\sigma(\beta)\,=\,\gamma$. Hence $G$ is two transitive.

Conversely if $G$ is two transitive, then there is a $\sigma\,\in \,G$ such that 
$\sigma(\alpha)\,=\,\alpha$ and $\sigma(\beta)\,=\,\gamma$. Hence $\sigma\,\in\, \Gal(k(\widetilde{Y})/k(X)(\alpha))$. Since
$\beta$ and $\gamma$ were arbitrary roots, it follows that $F_1(z)$ is irreducible.
\end{proof}

\section{Genuinely ramified morphisms}

A covering map $f\, :\, Y\, \longrightarrow\, X$ between connected smooth projective curves is
called \textit{genuinely ramified} if any (all) of the following three equivalent conditions holds:
\begin{enumerate}
\item the induced homomorphism of \'etale fundamental groups $f_*\ :\, \pi_1^{\rm et}(Y)\, \longrightarrow\, \pi_1^{\rm et}(X)$
is surjective.

\item The fiber product $Y\times_X Y$ is connected.

\item $\dim H^0(Y, \, f^*f_*{\mathcal O}_Y)\,=\, 1$.
\end{enumerate}
(See \cite[p.~12828, Proposition 2.6]{BP} and \cite[p.~12830, Lemma 3.1]{BP}.)

\begin{theorem}\label{main}
Let $f\, :\, Y\, \longrightarrow\, X$ be a genuinely ramified morphism between
irreducible smooth projective curves. Then $Y'$ (defined in \eqref{g0}) is connected.
\end{theorem}

\begin{proof} We will show this by separating the cases when $f$ is Galois and when it is not. 

{\bf Case 1}:\, Assume that the map $f$ is Galois. Let

$$
\Gamma\, :=\, \text{Aut}(Y/X) \, \subset\, \text{Aut}(Y)
$$
be the Galois group of $f$. For each $\gamma\, \in\, \Gamma$,
\begin{equation}\label{e2}
Y_\gamma\,\, :=\,\, \{(y,\, \gamma. y) \,\, \big\vert\,\, y\, \in\, Y\}\,\,\subset\,\, Y\times_X Y
\end{equation}
is evidently an irreducible component isomorphic to $Y$. It is straightforward to check that these are all the irreducible
components of $Y\times_XY$. In other words, the irreducible components of $Y\times_XY$ are parametrized by $\Gamma$.

Let $G_f$ denote the dual graph of $Y\times_X Y$ which is constructed as follows. The 
vertices of $G_f$ are the irreducible components of $Y\times_X Y$, so the vertices of $G_f$ are indexed by the elements
of $\Gamma$. For any two vertices $v_1$ and $v_2$ of $G_f$ there is at most one edge joining them, and there is an
edge joining $v_1$ and $v_2$ if and only if the irreducible components of $Y\times_X Y$ corresponding to $v_1$ and $v_2$
intersect. It is easy to see that the Galois group $\Gamma$ acts naturally on $G_f$. The action of $\Gamma$ on the vertices of $G_f$
coincides with the left-translation action of $\Gamma$ on itself (recall that the vertices are parametrized by the
elements of $\Gamma$). In particular, the action of $\Gamma$ on the vertices of $G_f$ is free and
transitive.

Since $Y\times_X Y$ is connected, so is $G_f$. The diagonal $\Delta$ in \eqref{g2} is $Y_e$ (see \eqref{e2}), where
$e\,\in\, \Gamma$ is the identity element. For any element $g\, \in\, \Gamma$, let $G^g_f$
be the graph obtained from $G_f$ by deleting the vertex corresponding
to $g$ together with all the edges that contain this vertex corresponding to $g$.
So the closure $\overline{(Y\times_X Y) \setminus \Delta}$ of $(Y\times_X Y) \setminus \Delta\, \subset\, Y\times_X Y$
is connected if and only if the graph $G^e_f$ is connected.

Assume that $Y\times_X Y \setminus \Delta\,=\, Y\times_X Y\setminus Y_e$ is disconnected. As noted above, this implies that
$G^e_f$ is disconnected. Since the action of $\Gamma$ on the vertices of $G_f$ is transitive, we conclude that
the graph $G^g_f$ is disconnected for each $g\,\in\, \Gamma$. 

\textbf{Claim 1.}\, Given a finite connected graph $\mathbb G$, let $v$ and $w$ be two vertices of $\GG$ such that distance between
$v$ and $w$ is the largest 
among all pairs of vertices in $\GG$. Then the graph $\GG_v$ obtained from $\mathbb G$, by deleting $v$ together with also all edges 
containing $v$, is connected.

We will prove the claim by showing that all the vertices of $\GG_v$ are connected to $w$. 
If $w'$ is a vertex of $\GG$ different from $v$, then the hypothesis that the distance between $v$ and $w$ is the largest
implies that there is a path in $\GG$ connecting $w$ and $w'$ which does not pass through $v$. Hence $w'$ and $w$ are
connected in $\GG_v$.

Hence the closure of $(Y\times_X Y) \setminus \Delta\, \subset\, Y\times_X Y$ is connected if $f$ is Galois.

{\bf Case 2.}\, Assume that $f$ is not Galois. 

Let
\begin{equation}\label{wf}
\widetilde{f}\,\,:\,\, \widetilde{Y}\,\,\longrightarrow\,\, X
\end{equation}
be the Galois closure of $f$ with
$h'\,:\,\widetilde{Y}\, \longrightarrow\, Y$ being the map through which $\widetilde f$ factors, meaning
$\widetilde{f}\,=\,f\circ h'$. It is known that $f$ is genuinely ramified does not imply $\widetilde f$ is genuinely ramified. 
Let $g\,:\,Z\,\longrightarrow\, X$ 
be the maximal \'etale cover dominated by $\widetilde f$, and let $h\,:\,\widetilde{Y}\,\longrightarrow\, Z$ be the factor map;
so $\widetilde{f}\,=\,g\circ h$. Let $W$ be the fiber product $Z\times_X Y$, and also
denote by $p_Z$ (respectively, $p_Y$) the projection map from $W$ to $Z$ (respectively, $Y$). Note that since $g:Z\,\longrightarrow\,
X$ is \'etale, the map $p_Y$ is also \'etale. Moreover, since $Y$ is smooth, the fiber product $W$ is also smooth.
Consequently, we get the following commutative diagram:
\begin{equation}\label{d}
\xymatrix{
& \widetilde Y\ar[d]^q\ar@/_1.5pc/[ddl]_h\ar@/^1.5pc/[ddr]^{h'}\\
& W\ar[dl]_{p_Z}\ar[dr]^{p_Y}\\
Z\ar[dr]_g & & Y\ar[dl]^f\\
& X
}
\end{equation}

Since $f$ is genuinely ramified and $g$ is \'etale, we have $k(Z)\cap k(Y)\,=\,k(X)$. Also note that $k(Z)$ is Galois over $k(X)$. This is because the Galois closure of an \'etale cover is \'etale and
$g$ is the maximal intermediate \'etale cover of the ramified Galois cover $\widetilde f$. So $W$ is an irreducible smooth curve and
the maps $q$, $p_Y$, $h'$ and $h$ are all Galois covers. Since $g$ is the maximal \'etale cover dominated by $\widetilde f
\,:\, \widetilde Y \,\longrightarrow\, X$, the map 
$h$ is genuinely ramified. Consequently, $p_Z$ --- being dominated by $h$ --- is also genuinely ramified.

Since $p_Z$ is genuinely ramified, the fiber product $W\times_Z W$ is connected. Let $$D_{W/Z}
\, \subset\, W\times_Z W$$ be the diagonal, so $D_{W/Z}$ is identified with $W$. Consider the map
$$
\Phi\, \,:=\,\, (p_Y\times p_Y)\big\vert_{W\times_Z W}\, :\,\, W\times_Z W \,\longrightarrow\,Y\times_X Y.
$$

\textbf{Claim 2.}\, The inverse image, under this map $\Phi$, of $\Delta\, \subset\, Y\times_X Y$ coincides with $D_{W/Z}$.

It can be seen that to prove the above claim it suffices to prove that the degree of the map $\Phi\,:\,W\times_Z W\,\longrightarrow\, Y\times_X Y$
coincides with the degree of $p_Y\,:\,W\,\longrightarrow\, Y$. Indeed, as $D_{W/Z}$ is a component of $\Phi^{-1}(\Delta)$,
we have $D_{W/Z}\,=\,\Phi^{-1}(\Delta)$ if degrees of $p_Y$ and $\Phi$ coincide. Now note that $W\times_Z W$ is the fiber product
of $Y\times_X Y\,\longrightarrow\, X$ and $g\,:\,Z\,\longrightarrow\, X$ with $\Phi\,:\,W\times_Z W\,\longrightarrow\, Y\times_X Y$
being the first projection. Hence we have $\degree(\Phi)\,=\,\degree(g)\,=\,\degree(p_Y)$. This proves the claim.

{}From Claim 2 it follows that $(Y\times_X Y)\setminus \Delta$ is connected if $W\times_Z W\setminus D_{W/Z}$ is connected (note that $\Phi$ is
surjective).

Let $G$ (respectively, $H$) be the Galois group of $h$ (respectively, $q$); see \eqref{d}.
Then $H\,\le \,G$ is a subgroup (need not be normal). Note
that the preimage of $D_{W/Z}$ under the morphism $\widetilde Y\times_Z\widetilde Y\,\longrightarrow\, W\times_Z W$ is $\widetilde Y
\times_W \widetilde Y$. Hence to show that $(W\times_Z W)\setminus D_{W/Z}$ is connected it is enough to show that $(\widetilde Y
\times_Z\widetilde Y)\setminus \widetilde Y \times_W\widetilde Y$ is connected. Let $\Gamma_G$ (respectively, $\Gamma_H$)
be the dual graph of $\widetilde Y\times_Z\widetilde Y$ (respectively, $\widetilde Y\times_W\widetilde Y$). Then $\Gamma_H$ is
a subgraph of $\Gamma_G$. It will follow that $\widetilde Y\times_Z\widetilde Y\setminus \widetilde Y \times_W\widetilde Y$ is
connected once we are able to show that $\Gamma_G\setminus \Gamma_H$ is connected.

As $h$ is genuinely ramified the graph $\Gamma_G$ is connected. Also $G$ acts faithfully and transitively on $\Gamma_G$. The
cosets of $H$ define a 
partition of $\Gamma_G$ with $\Gamma_H$ as one of the parts. Let $\Gamma$ be the quotient graph of $\Gamma_G$ with respect to this partition. 
Then the $H$-cosets in $G$ are the vertices of $\Gamma$. Since $\Gamma_G$ is connected, so is $\Gamma$. Also, as $G$ acts on the set of 
$H$-cosets transitively, the action of $G$ on $\Gamma$ is transitive. Since $\Gamma$ is connected there exists a vertex $gH$ in $\Gamma$ such that 
$\Gamma\setminus \{gH\}$ is connected (see Claim 1). As $G$ acts transitively on $\Gamma$, it now follows that $\Gamma\setminus\{eH\}$ is
connected. But this implies that $\Gamma_G\setminus \Gamma_H$ is connected.
\end{proof}

\begin{corollary}\label{S_d-cover}
Let $f\,:\,Y\,\longrightarrow\, X$ be a genuinely ramified cover of degree $d$ between irreducible smooth
projective curves such that $f$ is a Morse map. Then the Galois group of the Galois closure of $f$ is full $S_d$.
\end{corollary}

\begin{proof}
Note that $(Y\times_X Y)\setminus \Delta$ is connected by Theorem \ref{main}. Lemma \ref{morse} implies that
$(Y\times_X Y)\setminus \Delta$ is smooth and hence it is irreducible. Lemma \ref{2-transitive} says that the Galois group $G$ of
the Galois closure of $f$ is a two transitive subgroup of $S_d$. Since the ramification type of $f$ is $(2,\,1,\,1,\,
\ldots,\,1)$, the group $G$ contains a transposition (\cite[Cycle lemma, page 95]{abh} or \cite[Lemma 2.4]{juul}). Hence we have $G\,=\,S_d$ \cite[Lemma 4.4.3]{serre}.
\end{proof}

\begin{corollary}
Let $f\,:\,Y\,\longrightarrow\, X$ be a genuinely ramified cover of degree $d$ between irreducible smooth curves such
that $f$ is a Morse map. Then $Y'$ in \eqref{g0} is smooth irreducible, and $$ q'_1\,:\,Y'\,\longrightarrow\, Y$$
(see Lemma \ref{morse} for $q'_1$) is a genuinely ramified Morse map of degree $d-1$. 
\end{corollary}

\begin{proof}
As in \eqref{wf}, let $\widetilde f\,:\,\widetilde Y\,\longrightarrow\, X$ be the Galois closure of $f$. By
Corollary \ref{S_d-cover}, this $\widetilde f$ is a $S_d$-Galois cover.
Lemma \ref{morse} implies that $Y'$ is smooth and Theorem \ref{main} implies $Y'$ is connected, and hence $Y'$
is irreducible. Consequently, by Lemma \ref{morse}, the covering morphism $q'_1$ is a Morse map of degree $d-1$.

Note that $h'\,:\,\widetilde{Y}\,\longrightarrow\, Y$ in \eqref{d} is a $S_{d-1}$-Galois cover which dominates $q'\,:\,Y'\, \longrightarrow\, 
Y$. This follows by noticing that $\Gal(k(\widetilde{Y})/k(Y))$ is an index $d$ subgroup of $\Gal(k(\widetilde{Y})/k(X))$ and the fact that 
any subgroup of $S_d$ of index $d$ is actually $S_{d-1}$. It is standard fact that there is no proper subgroup of $S_{d-1}$ that properly
contains $S_{d-2}$. This implies that there is no field in between $k(Y)$ and $k(Y')$. Hence $q'_1:Y'\, \longrightarrow\, Y$ does not 
admit any intermediate covers. So if we show that $q'_1$ is ramified then it is genuinely ramified.

Note that if $d\,=\,2$ then $q'_1$ is degree 1 and there is nothing more to show. If $d\,\ge\, 3$, then note that $h'$ is ramified. Indeed, for a 
branch point $x\,\in\, X$ of $f$, every point of $\widetilde{f}^{-1}(x)$ is a ramification point of $\widetilde f$. Hence every point in 
$f^{-1}(x)$ at which $f$ is unramified is a branch point of $h'$. Since $f$ is a Morse map and the degree $f$ is at least 3, it follows that 
$h'$ is a ramified cover. Also note that $h'$ is the Galois closure of $q_1'$ as the degree of $q_1'$ is $d-1$. Hence $q_1'$ is also 
ramified.
\end{proof}

\begin{definition}\label{non-flex}
Let $X\,\subset\, \PP^n$ be a smooth curve of degree $d$ not contained in any hyperplane. A hyperplane $H\,\subset\, \PP^n$
is said to be an \textit{ordinary tangent} of $X$ at $x\, \in\, X$ if $x\, \in\, H$, the multiplicity of the intersection of
$H$ and $X$ at $x$ is two and their intersection is transversal at points of $X\cap H \setminus \{x\}$. We will say a point $x\in X$ is
\textit{non-flex} if there is an hyperplane $H\,\subset\, \PP^n$ which is an ordinary tangent of $X$ at $x$. 
\end{definition}

The dual variety $$\check{X}\,\,:=\,\,\{H\,\in\, \check{\PP}^n\,\,\mid\,\, H\, \text{ does not intersect 
$X$ transversally} \}$$ is a hypersurface in $\check{\PP}^n$.

\begin{lemma} \label{ord-tan}
Let $X\,\subset\, \PP^n$ be a smooth curve of degree $d$ not contained in any hyperplane. Assume that $X$ has a non-flex point.
There is an open dense subset $U$ of $\check{X}$ such that every $H$ in $U$ is an ordinary tangent of $X$.
\end{lemma}

\begin{proof}
The inclusion map $X\, \hookrightarrow\, \PP^n$ will be denoted by $\iota_X$. Let
\begin{equation}\label{f1}
\PP^n\times\check{X}\,\, \supset\,\, {\mathcal H}\,\, \stackrel{\varpi'}{\longrightarrow}\,\, \check{X}
\end{equation}
be the tautological hyperplane bundle over $\check{X}$. Let
\begin{equation}\label{f2}
\phi\,\,:\,\, {\mathcal H}\,\, \longrightarrow\,\, \PP^n
\end{equation}
be the natural projection. Define ${\mathcal X}\,:=\, \phi^{-1}(X)\, \subset\, {\mathcal H}$. The
inclusion map ${\mathcal X}\, \hookrightarrow\, {\mathcal H}$ will be denoted by $\iota_{\mathcal X}$. Let
\begin{equation}\label{f4}
\varphi\, \, :\,\, {\mathcal X}\,\, \longrightarrow\,\, X
\end{equation}
be the natural projection. Note that $\varphi$ is given by the restriction, to $\mathcal X$, of
the map $\phi$ in \eqref{f2}. Define
\begin{equation}\label{f3}
\varpi\,\, :\,\, {\mathcal X}\,\, \longrightarrow\,\, \check{X}
\end{equation}
to be the restriction of $\varpi'$ in \eqref{f1}. Note that $\varpi$ is a finite morphism.

Let $T_{\varpi'}\, \subset\, T{\mathcal H}$ be the relative tangent bundle for the projection $\varpi'$ in \eqref{f1}.
The natural injective homomorphisms
$$
\iota^*_{\mathcal X} T_{\varpi'}\,\longrightarrow\, (\iota_X\circ\varphi)^* T\PP^n
 \ \ \,\text{ and }\ \ \, \varphi^*TX \, \longrightarrow\, (\iota_X\circ\varphi)^* T\PP^n
$$
together produce a homomorphism
$$
\Psi\,\, :\,\, \iota^*_{\mathcal X} T_{\varpi'}\oplus \varphi^*TX \,\, \longrightarrow\,\, (\iota_X\circ\varphi)^* T\PP^n.
$$
From the assumption that $X$ has a non-flex point it follows immediately that $\Psi$ is surjective over some point. Let
$$
{\mathcal W}\,\, \subset\,\, {\mathcal X}
$$
be the nonempty Zariski open subset over which $\Psi$ is an isomorphism. Consider the Zariski open subset
$$
\mathcal{U}\,\, :=\,\, \varpi({\mathcal W})\,\, \subset\,\, \check{X}.
$$

Take any $H_0\,\in\, \check{X}$ which is an ordinary tangent of $X$ at some point $x_0\, \in\, X$ (such an hyperplane exists by our
assumption). Then
$$
\varpi^{-1}(H_0)\cap {\mathcal W}\,\, =\,\, (X\cap H_0)\setminus \{x_0\}.
$$
In particular, $\varpi^{-1}(H_0)\cap {\mathcal W}$ consists of distinct $d-2$ reduced points of $X$. 
Note that $\varpi^{-1}(H_0)\,=\, (\varpi^{-1}(H_0)\cap {\mathcal W}) \cup 2x_0 \, \subset\, X$. Also, note that for any
$H\, \in\, {\mathcal U}$, the inverse image $\varpi^{-1}(H)\cap {\mathcal W}$ is a reduced subscheme of $X$ whose length
is at most $d-2$.

Let $U\,\, \subset\,\, {\mathcal U}$ be the locus of all points $H\, \in\, {\mathcal U}$ such that the length of
$\varpi^{-1}(H)\cap {\mathcal W}$ is $d-2$ (the maximal possible).

The hyperplane in $\PP^n$ corresponding to each point of $U$ is an ordinary tangent of $X$ at some point.
Also, $U$ is an open subset of $\check{X}$.
\end{proof}

Under the hypothesis of the Lemma \ref{ord-tan}, let $L$ be a line in $\check{\PP}^n$ that intersects $\check{X}$ transversely, $L\cap \check{X}$ is contained in $U$ and $X\cap H_1\cap H_2$ is empty for $H_1\,\ne\, H_2\,\in\, L$. Note that a general $L$ in $\check{\PP}^n$ will have this property.
We define a map
\begin{equation}\label{ef}
f\,:\,X\,\longrightarrow\, L
\end{equation}
as follows. For $x\,\in \,X$, let $H_x$ be the set of hyperplanes
in $\PP^n$ passing through $x$. Then
$H_x$ defines a hyperplane in $\check{\PP}^n$ not containing $L$ and hence it intersects $L$ at a unique point. The map $f$ sends $x$ to this point of $L$.

The following is a generalization of \cite{Yo} to any smooth irreducible curve in $\PP^n$ and on the other hand to any characteristic.

\begin{corollary}\label{cor-l}
Let the hypothesis on $X$ be as in Lemma \ref{ord-tan} and $f:X\to L$ be as in \eqref{ef}.
The projection of $X$ from the codimension two linear subspace in $\PP^n$ corresponding to $L$ coincides with the map $f$. The map $f$ is Morse and has Galois group $S_d$.
\end{corollary}

\begin{proof}
Let $H_1$ and $H_2$ be two distinct points of $L$; then $A\,=\,H_1\cap H_2$ is a codimension two linear subspace of $\PP^n$. The map
$f\,:\, X\,\longrightarrow\, L$ is the projection of $X$ from $A$.

By the description of $f$, if $H\in L$ is not in $\check{X}$ then $H$ intersects $X$ transversally and hence $f^{-1}(H)$ has $d$ distinct points. So $f$ is not branched
at $H$. If $H\in L\cap U$ then $H$ intersects $X$ transversely at $d-2$ distinct points and and $H$ intersects $X$ with multiplicity two at one point (say $x$). 
Hence $f$ is \'etale at all points of $f^{-1}(H)$ except at $x$ and $x$ is a double point of $f$. Hence $f$ is Morse, since $L\cap \check{X}$ is contained in 
$U$.
 
Since $\PP^1$ is simply connected $f$ is also genuinely ramified. Hence by Corollary \ref{S_d-cover}, it follows that the Galois group of $f$ is $S_d$.
\end{proof}

\begin{remark}
 Note that the hypothesis that $X\,\subset\,\PP^2$ contains a non-flex point is equivalent to the dual map from $X$ to $\check{X}$ is separable if the characteristic of the base field is odd (see \cite[Proposition 1.5]{pardini}). 
\end{remark}

\section*{Acknowledgements}

We are very grateful to the referee for comments which helped in removing an error and improving the exposition.

\end{document}